\documentclass[11pt, a4paper]{article}

\usepackage[T1]{fontenc}
\usepackage{amsmath, amssymb, amsthm}
\usepackage[margin=1in]{geometry}
\usepackage{booktabs} 
\usepackage{hyperref}
\usepackage{authblk}

\hypersetup{
    pdftitle={A Constructive Heuristic Sieve for the Twin Prime Problem},
    pdfauthor={Yuhang Shi},
    pdfsubject={Number Theory},
    pdfkeywords={twin primes, sieve theory, prime numbers, f(t;z) analysis},
    colorlinks=true,
    linkcolor=blue,
    urlcolor=blue,
    citecolor=red,
    unicode=true
}

\newtheorem{theorem}{Theorem}[section]

\newtheorem{proposition}[theorem]{Proposition}

\theoremstyle{definition}
\newtheorem{definition}[theorem]{Definition}

\theoremstyle{remark}

\title{\bf A Constructive Heuristic Sieve for the Twin Prime Problem}
\author{
    Yuhang Shi \\
    \small \texttt{lostangel1964@email.phoenix.edu} \\
    \small Xi'an Qing'an Senior High School
}

\date{July 1, 2025}

\begin{document}

\maketitle

\begin{abstract}
The quantitative distribution of twin primes remains a central open problem in number theory. This paper develops a heuristic model grounded in the principles of sieve theory, with the goal of constructing an analytical approximation for the twin prime constant from first principles. The core of this method, which we term ``$f(t; z)$ function analysis,'' involves representing the sieve's density product as a ratio of infinite series involving $f(t;z)$, the elementary symmetric polynomials of prime reciprocals. This framework provides a constructive path to approximate the celebrated Hardy-Littlewood constant for twin primes. We present a detailed and transparent numerical analysis based on verifiable code, comparing the truncated series approximation to empirical data. The limitations of the model, particularly a systematic overestimation and its dependence on series truncation, are rigorously discussed. The primary value of this work lies not in proposing a superior predictive formula, but in offering a clear, decomposable, and analytically tractable heuristic for understanding the multiplicative structure of sieve constants.
\end{abstract}

\section{Introduction}

The study of prime numbers has long captivated mathematicians. A twin prime pair is a pair of primes $(p, p+2)$. The Twin Prime Conjecture, which posits their infinitude, represents a specific case of de Polignac's broader conjecture from 1849 \cite{Polignac1849}. While the Prime Number Theorem furnishes a definitive asymptotic for the prime counting function $\pi(x)$ \cite{Hadamard1896, Poussin1896}, the corresponding result for the twin prime counting function, $\pi_2(x)$, remains elusive. The celebrated Hardy-Littlewood conjecture \cite{Hardy1923} offers a precise prediction:
$$ \pi_2(x) \sim 2C_2 \frac{x}{(\ln x)^2}, \quad \text{where} \quad C_2 = \prod_{p>2, p \text{ prime}} \left(1 - \frac{1}{(p-1)^2}\right) $$
is the twin prime constant.

Sieve methods, originating with Brun \cite{Brun1919}, provide powerful upper bounds but have not yet resolved the conjecture. However, the field has seen stunning progress, highlighted by the landmark result of Yitang Zhang, who proved the existence of a finite bound on gaps between primes \cite{Zhang2014}, a result rapidly improved upon by others \cite{Maynard2015}.

This paper charts an alternative heuristic course. We construct a model for $\pi_2(x)$ from the ground up, based on a sieve-theoretic framework. The core of our methodology, termed ``$f(t; z)$ function analysis,'' represents the sieve correction factor as a ratio of infinite series involving sums of prime reciprocals. This work culminates in our main result, the following heuristic approximation for the twin prime counting function:
\begin{equation}
    \pi_2(x) \sim \frac{2x}{(\ln x)^2} \cdot \frac{\sum_{t=0}^{\infty} (-2)^t f(t;z)}{\left(\sum_{t=0}^{\infty} (-1)^t f(t; z\right)^2}.
\end{equation}
The derivation of this formula, a detailed and candid assessment of its performance via truncated series, and a critical analysis of its intrinsic limitations form the central focus of this paper.

\section{Heuristic Sieve Framework}

A naive probabilistic argument suggests $\pi_2(x) \approx x/(\ln x)^2$. Sieve methods correct for this by analyzing the divisibility of $n(n+2)$ by small primes. For a prime $p>2$, for $n$ and $n+2$ to be a potential prime pair, we require $n \not\equiv 0 \pmod p$ and $n \not\equiv -2 \pmod p$. The standard heuristic for the density of twin primes leads to the twin prime constant $2C_2$:
$$ 2C_2 = 2 \prod_{p>2} \frac{p(p-2)}{(p-1)^2} = 2 \prod_{p>2} \frac{1-2/p}{(1-1/p)^2}. $$
Our model aims to approximate this constant by truncating the products at a sieving limit $z$ and analyzing the resulting expression.

\begin{proposition}[Heuristic Sieve Model]
We model the twin prime counting function as:
\begin{equation} \label{eq:model_base_prop}
\pi_2(x) \sim \mathcal{D}(z) \frac{x}{(\ln x)^2}, \quad \text{where} \quad \mathcal{D}(z) = 2 \prod_{3 \le p \le z} \frac{1-2/p}{(1-1/p)^2}.
\end{equation}
The core of our work is to develop a precise analytical approximation for the correction factor $\mathcal{D}(z)$ using our ``$f(t;z)$ function analysis.''
\end{proposition}

\section{The \texorpdfstring{$f(t;z)$}{f(t;z)} Function Analysis}

Our central method is to analyze the numerator and denominator of $\mathcal{D}(z)$ by expanding them into infinite series.

\begin{definition}[The $f(t;z)$ Functions]
Let $z \in \mathbb{R}^+$. We define $f(t;z)$ as the elementary symmetric polynomial of degree $t$ in the variables $\{1/p \mid p \text{ is an odd prime and } p \le z\}$.

\begin{equation}
    f(t;z) = \sum_{3 \le p_1 < p_2 < \dots < p_t \le z} \frac{1}{p_1 p_2 \dots p_t}.
\end{equation}

By convention, $f(0;z) = 1$.
\end{definition}

By expanding the product forms, we can formally express the components of $\mathcal{D}(z)$ as series:
$$ \text{Numerator Term:} \quad \prod_{3 \le p \le z} \left(1 - \frac{2}{p}\right) = 1 - 2\sum_{3 \le p \le z}\frac{1}{p} + 4\sum_{3 \le p_1<p_2 \le z}\frac{1}{p_1 p_2} - \dots= \sum_{t=0}^{\infty} (-2)^t f(t;z). $$
$$ \text{Denominator Term:} \quad \prod_{3 \le p \le z} \left(1 - \frac{1}{p}\right) = 1 - \sum_{3 \le p \le z}\frac{1}{p} + \sum_{3 \le p_1<p_2 \le z}\frac{1}{p_1 p_2} - \dots = \sum_{t=0}^{\infty} (-1)^t f(t;z). $$
This transforms the problem of estimating a product into one of summing series. The asymptotic behavior of $f(t;z)$ for large $z$ and fixed $t$ is known to be $f(t;z) \approx \frac{1}{t!}(\ln(\ln z)+M')^t$, where $M'$ is the Meissel-Mertens constant for odd primes. This asymptotic behavior provides theoretical insight, suggesting the series converge. However, it is crucial to note that this asymptotic is not used for our numerical calculations, which rely solely on the exact values of $f(t;z)$ to avoid approximation errors.

\begin{theorem}[Recursive $f(t)$ Analysis]
The asymptotic behavior of $f(t;z)$ for large $z$ can be determined recursively. Let $L(z) = \ln(\ln z)$ and $M' = M - 1/2 \approx -0.0718$ be the Meissel-Mertens constant for odd primes.
\begin{itemize}
    \item \textbf{Base Case ($t=1$):} $f(1;z) = \sum_{3 \le p \le z} 1/p \sim L(z) + M'$.
    \item \textbf{Recursive Step:} A key insight is that $f(t;z)$ can be related to $f(t-1;z)$.
    \begin{equation}
    f(t;z) = \sum_{p \le z} \frac{1}{p} f(t-1; p-1) \sim \frac{1}{t} \left( f(1;z)f(t-1;z) - \sum_{p \le z}\frac{f(t-2;p-1)}{p^2} \right).
    \end{equation}
\end{itemize}
This leads to the leading-order approximation $f(t;z) \approx \frac{1}{t!}(L(z)+M')^t$, where the combinatorial factor $1/t!$ arises naturally from the recursive structure to account for unordered sets of primes.
\end{theorem}

\begin{proof}[Sketch of Recursive Derivation for $f(2;z)$]
We can write $f(2;z)$ as $\sum_{p_i \le z} \frac{1}{p_i} \left( \sum_{p_j < p_i} \frac{1}{p_j} \right)$. This avoids overcounting.
$$ f(2;z) = \sum_{p \le z} \frac{1}{p} f(1; p-1) \sim \sum_{p \le z} \frac{1}{p} (L(p) + M') $$
Alternatively, and more simply, the identity $2f(2;z) = f(1;z)^2 - \sum_{p \le z} 1/p^2$ holds.
Let $S_k(z) = \sum_{3 \le p \le z} p^{-k}$. Then $2f(2;z) = S_1(z)^2 - S_2(z)$. Asymptotically, $S_1(z) \sim L(z)+M'$ and $S_2(z)$ converges to the prime zeta function for odd primes, $p_{odd}(2) = p(2) - 1/4$.
$$ f(2;z) \sim \frac{1}{2}((L(z)+M')^2 - p_{odd}(2)). $$
This recursive and combinatorial approach can be extended for all $t$, confirming the leading-order behavior $f(t;z) \sim \frac{1}{t!}(L(z)+M')^t$.
\end{proof}

\section{Corrected Main Formula and Numerical Analysis}

Our revised and mathematically consistent main formula directly incorporates the structure of $\mathcal{D}(z)$.

\begin{theorem}[Corrected Heuristic Formula]
The twin prime counting function can be approximated by:
\begin{equation}\label{eq:main_formula_theorem}
    \pi_2(x) \sim \frac{2x}{(\ln x)^2} \cdot \frac{\sum_{t=0}^{\infty} (-2)^t f(t; z)}{\left(\sum_{t=0}^{\infty} (-1)^t f(t; z)\right)^2},
\end{equation}
where $z$ is a suitably chosen sieving limit, such as $z=x^{1/4}$.
\end{theorem}

\subsection{Numerical Verification and Analysis}

We now test this model with data generated from verifiable code. For each value of $x$, we set $z = \lfloor x^{1/4} \rfloor$, generate the list of odd primes up to $z$, and compute the exact values of $f(t;z)$ for $t \in [0, 4]$. The series are truncated at $t=4$ to compute an approximate correction factor $\mathcal{D}_{\text{approx}}(z)$, which is then used to calculate the final approximation $\pi_{2,\text{approx}}(x)$. The results are presented in Table \ref{tab:comparison}.

\begin{quote}
\small
\textbf{Code Availability and Reproducibility:} 
The complete computational implementation of this verification framework, 
including all algorithms and data generation routines, is permanently archived 
via Zenodo \cite{shi_pi2_2025}. This DOI-resolved record 
contains:
\begin{itemize}
    \item Executable Jupyter notebooks for interactive verification
    \item Version-controlled Python scripts (MIT licensed)
    \item Dependency specifications for full reproducibility
\end{itemize}
Researchers may independently validate all tabulated results through 
documented procedures at \url{https://doi.org/10.5281/zenodo.15857682}.
\end{quote}

\begin{table}[htbp]
\centering
\caption{Comparison of Twin Prime Counting Function Approximations (Verifiable Data)}
\label{tab:comparison}
\begin{tabular}{@{}lrrrrr@{}}
\toprule
$x$ & $\pi_2(x)$ (True) & $\pi_{2,\text{HL}}(x)$ & Rel. Err. (HL) & $\pi_{2,\text{approx}}(x)$ & Rel. Err. (This work) \\
\midrule
$10^4$ & 205 & 214 & +4.4\% & 161 & -21.5\% \\
$10^5$ & 1,224 & 1,249 & +2.0\% & 1,087 & -11.2\% \\
$10^6$ & 8,169 & 8,167 & -0.02\% & 11,978 & +46.6\% \\
$10^7$ & 58,980 & 58,754 & -0.38\% & 163,740 & +177.6\% \\
\bottomrule
\end{tabular}
\end{table}

\paragraph{Analysis of Results:}
The verifiable data in Table \ref{tab:comparison} reveal the true performance of our heuristic model.
\begin{enumerate}
    \item \textbf{Systematic Overestimation:} Unlike the erroneous data presented in a preliminary version of this work, the correct results show that our model, with the chosen parameters ($z=x^{1/4}$, truncation at $t=4$), systematically overestimates the twin prime count for larger $x$. The relative error is significant and grows with $x$.
    \item \textbf{Source of Discrepancy:} The overestimation likely arises from the truncation of the series. The correction factor $\mathcal{D}(z)$ is an approximation of $2C_2 \approx 1.32$. Our calculation shows that the truncated series yields a value for $\mathcal{D}_{\text{approx}}(z)$ that is larger than $2C_2$ and increases with $z$. For example, for $x=10^6$, our model calculates $\mathcal{D}_{\text{approx}}(31) \approx 1.91$, which is much larger than $1.32$. This indicates that the higher-order terms ($t>4$) in the series are crucial for reducing the value of the correction factor towards $2C_2$.
\end{enumerate}

\subsection{Model Limitations and Path Forward}
The primary value of this model is not its predictive accuracy in this truncated form, but its analytical structure. The discrepancy highlights the challenges of heuristic models:
\begin{itemize}
    \item \textbf{Truncation Sensitivity:} The model is highly sensitive to the truncation point. A more sophisticated analysis, perhaps including more terms or estimating the tail of the series, is required for better accuracy.
    \item \textbf{Choice of Sieving Limit:} The parameter $z=x^{1/4}$ is a simple choice. The model's performance may vary significantly with different choices of $z(x)$.
\end{itemize}
This honest assessment of the model's performance, based on verifiable data, is crucial for scientific integrity.

\section{Conclusion}

This paper has developed a heuristic sieve model for the twin prime counting function using a methodology we call ``$f(t;z)$ function analysis.'' Our primary contribution is the presentation of a transparent, constructive framework for approximating the twin prime constant.

A rigorous numerical verification, based on open-source code, reveals that the straightforward truncation of the model's core series leads to a significant overestimation of the twin prime count, contrary to erroneous data presented in a preliminary draft. This finding does not invalidate the theoretical framework but rather highlights its sensitivity to truncation. The model accurately shows that the twin prime density is governed by a complex interplay of prime reciprocals, but that a simple, low-order truncation is insufficient to capture the constant precisely.

The value of this work, therefore, lies in its methodology and transparency. It offers an analytical lens into the structure of sieve constants and serves as a clear case study on the limitations of simple heuristic models. Future work should focus on a more advanced analysis of the series' convergence and the optimal choice of the sieving limit $z$.

\bibliographystyle{plain}
\bibliography{main}

\end{document}